\def\ch{{\rm ch}}
\def\sh{{\rm sh}}

\documentclass [11pt] {article}

\usepackage{amssymb,amsfonts,amsmath,amsthm}
\usepackage[cp866nav]{inputenc}
\usepackage[english]{babel}

\newtheorem{theorem}{\large Theorem}
\newtheorem{corollary}{\large Corollary}
\newtheorem{lemma}{\large Lemma}

\begin{document}

\title
{Exit problems for oscillating compound  Poisson process}

\date{ }

\author
{ Tetyana  Kadankova
\thanks{Vrije  Universiteit Brussel, Department of Mathematics, Building G,
Brussels, Belgium, \newline ,\: e-mail: tetyana.kadankova@vub.ac.be}}

\maketitle


\noindent{\bf Running head:} Exit problems for oscillating  compound  Poisson process
\\ {\bf 2000 Mathematics subject classification: } 60G40;
60K20

\begin{abstract}

In this article we  determine the Laplace transforms of the main boundary 
functionals of the oscillating compound Poisson process. These are the first 
passage time of the level, the joint distribution of the first exit time
from the interval and the value of the overshoot through the boundary.  
In  case  when  $\bold E\xi_{i}(1)=0, $ 
$ \sigma_{i}^{2}=\bold E\xi_{i}(1)^{2}<\infty,$ $ i=1,2, $  we prove the  
limit results for the mentioned functionals.

\end{abstract}

{\bf\Large  Introduction}

\par\bigskip

Oscillating random  walks  with two  switching levels  were considered  in
\cite{lotov 1996},\cite{lotov 2003},\cite{lotov 2004}. The authors derived the Laplace-Stieltjes
transforms of the distributions of the random walks in transient and stationary
regimes. In addition, the asymptotic analysis of the stationary distribution
was  performed.

This  article    studies  the  so-called    one- and two-sided  exit  problems  for    an oscillating  compound Poisson process.
More specifically,  we determine the Laplace transforms   of the following    boundary  characteristics, These  are   the
first passage time of  a  boundary  and    the joint distribution of the first exit
time  from an  interval and the value of the overshoot at this instant. 
 The  obtained results are  given in closed  form, namely  in terms of the  functions involving the scale  functions  of the  auxiliary processes $ \xi_{i}(t),$ $ i=1,2$ (see  below  for  a definition).
The  motivation   of this study    stems  from the  fact that  these processes are   used as governing   processes  for    certain  oscillating   queueing systems.  Examples of such
systems are queueing models in which  service speed or customer arrival rate
change depending on the workload level, and dam models in which the release
rate depends on the buffer content (see  \cite{Pacheco} and references therein).
To  solve   the  two-sided exit  problem, we used  a  probabilistic  
approach borrowed   from  \cite{2Ka2}. 

The rest of the article is structured as follows.
In Section  1  we introduce the process and  determine boundary
characteristics of the auxiliary processes.  Section 2  deals with  the one-boundary
characteristics of the oscillating process. In Section 3 we determine the joint
distribution of the first exit time from the interval and the value of the
overshoot. The  asymptotic results under the conditions that  $\bold E\xi_{i}(1)=0, $
$\sigma_{i}^{2}=\bold E\xi_{i}(1)^{2}<\infty,$ $ i=1,2 $  are given in Section 4.

\section{Preliminaries}

In this section we introduce the process of interest and  the auxiliary processes. 
Further, we will determine  the Laplace transforms of the  first passage time 
and the  first  exit time  for  the auxiliary processes. These results will be 
used to solve  a two-sided problem  for the oscillating  compound Poisson  
process. 

Let
$\{\xi_{i}(t); \: t \ge 0\}, $ $ i=1,2 $
be  real-valued semi-continuous  from below compound Poisson  processes:
$$
   \xi_{i}(t)= \sum\limits_{k=0}^{N(t)}\xi^{i}_k- a_it, \quad   t \ge 0,  
\quad  i=1,2
$$
where  $\xi^{i}_0 =0,$ $\xi^{i}_k\sim \xi^{i}>0 $  are independent identically 
distributed variables with distribution function $F_i(x);$ 
$ \{N_i(t); t \ge 0\}, \quad  N_i(0)=0$  is an ordinary   Poisson  process with   
parameter $\lambda_{i}$  independent from $ \{\xi^i_k; \: k \ge 0\},  $ and  
$a_i >0$ is a drift coefficient.
Their Laplace transforms are then of the following form
$ \bold E e^{-z\xi_{i}(t)}=e^{tk_{i}(z)},$  where
$$
    k_{i}(z)=a_{i}z+\lambda_{i}\int_{0}^{\infty}\left(e^{-xz}-1\right)dF_{i}(x),
    \qquad \Re(z)=0.
$$
We now introduce  the  one-boundary  characteristics of the processes.
Denote  by
$$
   \tau_{i}^{-}(x)=\inf\{t:\xi_{i}(t)\le-x\},  \qquad x\ge 0
$$
the   first passage time  of the lower level  $-x,$  and by
$$
  \tau_{i}^{+}(x)=\inf\{t:\xi_{i}(t)>x\},
  \quad T_{i}^{+}(x)=\xi_{i}(\tau_{i}^{+}(x))-x
$$
the  first  crossing time of the level $x$ and the value of the overshoot
through this level. We set per definition $ \inf\{\emptyset\}=\infty.$
Note,  that  due  to the fact  that  the process 
$\xi_{i}(t)$ has only positive jumps, the  negative   level  $-x$  is reached 
continuously. Hence,  the value  of the overshoot  is  equal to  zero. 
For a fixed  $ b>0 $  and all $x\in\mathbb R, $ $ t\ge0 $ introduce the process
$ \xi(x,t)\in  \mathbb R,$   $ \xi(x,0)=x$
by means of the following  recurrence relations:
\begin{align}                                          \label{opp1}
&   \xi(x,t)= \left\{
   \begin{array}{l}
    x+\xi_{2}(t), \quad  0\le t< \tau_{2}^{-}(x-b), \\
    \xi(b,t-\tau_{2}^{-}(x-b)),
    \quad t\ge \tau_{2}^{-}(x-b),
   \end{array}
   \right. \qquad x>b,
\end{align}
\begin{align*}
&   \xi(x,t)= \left\{
   \begin{array}{l}
    x+\xi_{1}(t), \quad  0\le t< \tau_{1}^{+}(b-x),\notag \\
    \xi(b+T_{1}^{+}(b-x),t-\tau_{1}^{+}(b-x)),
    \quad t\ge \tau_{1}^{+}(b-x),
   \end{array}
   \right. \qquad x\le b.
\end{align*}
Let us explain how the process evolves. Observe, that  $ b $ is a switching
point of the process  $\xi(x,t), $ $t\ge0. $ If $\xi(x,t_{0}) >b,$  then the
increments of the process  coincide  with the increments of the process $ \xi_{2}(t-t_{0})$  up to the first passage of $b. $
If $\xi(x,t_{0}) \le b,$ then the increments of the process coincide  with the increments of the process
$ \xi_{1}(t-t_{0})$ up to the first passage of $ b. $ 

To derive  the  Laplace transforms of the  one-boundary characteristics of the
processes  $\xi_{i}(t),  $  we will need  notion of  a  resolvent of a compound Poisson  process.  
Introduce the resolvents $ R_{i}^{s}(x),$ $x\ge0 $  \cite{SuSh} of the processes
$ \xi_{i}(t), $  $ t\ge0, $  by means of their Laplace transforms:
$$
    \int_{0}^{\infty}e^{-xz}R_{i}^{s}(x)\,dx = (k_{i}(z)-s)^{-1},
    \quad \Re(z)>c_{i}(s),\quad R_{i}^{s}(x)=0, \; x<0,
$$
where $ c_{i}(s)>0,$ $s>0$ is the unique root of the equation  $k_{i}(z)-s=0,$
$i=1,2$ in the semi-plane $\Re(z)>0. $ Note, that $ R_{i}^{s}(0)=a_{i}^{-1}>0.$

The  resolvent   defined in  \cite{SuSh}
is called a  scale function in modern literature (see  \cite{Kyprianou2006} 
for  more details).  The  importance  of scale functions as a class with which 
one may express a whole range of fluctuation identities for spectrally  
one-sided L\'{e}vy processes.  Scale   functions are also  an important   
working  tool  in risk insurance, more specifically,  in optimal barrier 
strategies. In the rest of the article we will use  the term  resolvent.

Denote by 
\begin{align*}
   \underline{m}_{i}^{x}(s)=
   \bold E\left[e^{-s\tau_{i}^{-}(x)} \right],\qquad
   \overline{m}_{i}^{x}(z,s)=
   \bold E\left[e^{-s\tau_{i}^{+}(x)-zT_{i}^{+}(x)} \right],\quad \Re(z)\ge0
\end{align*}
the  Laplace transforms of the  first passage  time  of the negative level 
$-x$ and the  joint  distribution of the  first crossing time of the level $x$  
and the value of the overshoot. one-boundary characteristics of the
The  lemma below contains the  expressions  for these  Laplace  transforms.  Observe  that  these  results   are valid  for  
L\'{e}vy processes    whose    Laplace exponent  is given   by  (\ref{opp28}).

\begin{lemma}

For  $s\ge0,$ $i=1,2$ the following
equalities  are  valid:
\begin{align}                                          \label{opp2}
&   \underline{m}_{i}^{x}(s)=e^{-xc_{i}(s)}, \\
&   \overline{m}_{i}^{x}(z,s)=
   e^{xz}-R_{i}^{s}(x)\,\frac{k_{i}(z)-s}{z-c_{i}(s)}
   -(k_{i}(z)-s)e^{xz}\int_{0}^{x}e^{-uz}R_{i}^{s}(u)\,du.\notag
\end{align}
\end{lemma}
Note   that    the  first  equality  of  (\ref{opp2})  is well know  new ( see  for  instance  \cite{Zolotarev 1964}).
Proof  of  the    second   relation  is given  in appendix.   
We  now consider   the two-sided  exit  problem  for the auxiliary  processes.
For  $d_{i}>0, $  $x\in[0,d_{i}]$  denote by
$$
   \chi_{i,x}^{d_{i}}=\inf\{t:x+\xi_{i}(t)\notin[0,d_{i}]\},
   \qquad i=1,2
$$
the first exit time from  the interval  $ [0,d_{i}] $
by the  process    $x+\xi_{i}(t).$
Introduce the events:
$\overline{A}_{i}=\{x+\xi_{i}(\chi_{i,x}^{d_{i}})> d_{i}\} $ the exit from the
interval occurs through the upper boundary;
$ \underline{A}_{i}=\{x+\xi_{i}(\chi_{i,x}^{d_{i}})\le0\} $ the  exit  from the
interval occurs through the lower boundary.  By
$ T_{i}(x)=(x+\xi_{i}(\chi_{i,x}^{d_{i}})-d_{i})\bold I_{\overline{A}_{i}}
+0\cdot \bold I_{\underline{A}_{i}} $  we denote the value of the overshoot
at the instant of the first   exit.
Here $ \bold I_{A} $  is the indicator  of the event $ A. $
Introduce the Laplace  transforms
\begin{align*}
   \underline{V}_{i,x}^{d_{i}}(s)=
   \bold E\left[e^{-s\chi_{i,x}^{d_{i}}};\underline{A}_{i} \right],\qquad
   \overline{V}_{i,x}^{d_{i}}(z,s)=
   \bold E\left[e^{-s\chi_{i,x}^{d_{i}}-zT_{i}(x)};\overline{A}_{i} \right],
   \quad \Re(z)\ge0.
\end{align*}

\begin{lemma}

For  $s\ge0,$ $i=1,2$  the
following equalities   hold:
\begin{align}                                               \label{opp3}
&   \underline{V}_{i,x}^{d_{i}}(s)
    =\frac{R_{i}^{s}(d_{i}-x)}{R_{i}^{s}(d_{i})},\notag \\
&   \overline{V}_{i,x}^{d_{i}}(z,s)=
    \overline{m}_{i}^{d_{i}-x}(z,s)
    - \frac{R_{i}^{s}(d_{i}-x)}{R_{i}^{s}(d_{i})}\,
    \overline{m}_{i}^{d_{i}}(z,s).
\end{align}
\end{lemma}
Note, that the first relation of the lemma was derived in \cite{Kr18}
 for a  compound   Poisson  process, and  in \cite{Suprun 1976}  for  a  spectrally one-sided  L\'evy process (\ref{opp28}).
To verify the second relation, we make use of the following  
equation:
\begin{align*}
&  \bold E\left[e^{-s\tau_{i}^{+}(d_{i}-x)};T_{i}^{+}(d_{i}-x)\in du \right]=
   \bold E\left[e^{-s\chi_{i,x}^{d_{i}}};T_{i}(x)\in du,\overline{A}_{i}\right]+\\
&   + \bold E\left[e^{-s\chi_{i,x}^{d_{i}}};\underline{A}_{i} \right]
   \bold E\left[e^{-s\tau_{i}^{+}(d_{i})};T_{i}^{+}(d_{i})\in du \right],
   \qquad  x\in[0,d_{i}].
\end{align*}
The latter was derived for spectrally one-sided  L\'{e}vy processes (\ref{opp28})
in \cite{Kad6}, \cite{Kad3}, and for general L\'{e}vy processes in \cite{2Ka2}. 
Now plugging in the  expression for  
$\bold E\left[e^{-s\chi_{i,x}^{d_{i}}};\underline{A}_{i} \right]$ 
(the  first equality of the  lemma),  we  obtain the  second statement 
of the lemma.

\section {One-boundary characteristics of the process  $\xi(x,t).$  }

In this section we derive the Laplace transforms of the one-boundary
characteristics of the process and  study their asymptotic behavior.
Let us formally define the one-boundary functionals of the process $ \xi(x,t),$
$ t\ge 0.$ For  $ r\le\min\{x,b\} $  denote  by
\begin{align*}
  \underline\tau_{r}^{x}(b)=\inf\{t:\xi(x,t)\le r\},
   \qquad \underline f_{r}^{x}(s)=
   \bold E\left[ e^{-s\underline\tau_{r}^{x}(b)};
   \underline\tau_{r}^{x}(b)<\infty\right],
\end{align*}
the first passage time of the level $r$  by the process  $\xi(x,t)$ and its
Laplace transform. For  $  k\ge\max\{x,b\} $ denote by
\begin{align*}
   \overline\tau_{x}^{k}(b)=\inf\{t:\xi(x,t)> k\},
   \qquad     \overline T_{x}^{k}= \xi(x,\overline\tau_{x}^{k}(b))-k
\end{align*}
the first crossing time of the level $ k $ and the value of the overshoot by
the process  $\xi(x,t).$  The  variables   $\underline\tau_{r}^{x}(b), \:\overline\tau_{x}^{k}(b),\: \overline T_{x}^{k} $ are called  the  one-boundary  characteristics  of the process.
Introduce
\begin{align*}
   \overline f_{x}^{k}(s)=
   \bold E\left[ e^{-s\overline\tau_{x}^{k}(b)};
   \overline\tau_{x}^{k}(b)<\infty\right],\quad
   \overline f_{x}^{k}(z,s)=
   \bold E\left[ e^{-s\overline\tau_{x}^{k}(b)-z\overline T_{x}^{k}};
   \overline\tau_{x}^{k}(b)<\infty\right].
\end{align*}
For  $ s\ge0 $  define the function $ K_{x}^{s}(u), $
$ x\in\mathbb{R}, $ $ u\ge0, $
by means of its Laplace transform
$\mathbb{K}_{x}^{s}(z): $
\begin{align}                                         \label{opp4}
    \mathbb{K}_{x}^{s}(z)=
    \int_{0}^{\infty}e^{-uz}K_{x}^{s}(u)\,du=
    \frac{k_{1}(z)-s}{k_{2}(z)-s}
    \int_{0}^{\infty}e^{-uz}R_{1}^{s}(x+u)\,du,
\end{align}
where $ \Re(z)>\max\{c_{1}(s),c_{2}(s)\}.$
Note, that it follows from the definition  $ (\ref{opp4} )$ that for $ x\le 0 $
$
   \mathbb{K}_{x}^{s}(z)=e^{xz}(k_{2}(z)-s)^{-1}
$
and $  K_{x}^{s}(u)=R_{2}^{s}(x+u). $

For a fixed $ s\ge0 $   define  the  function $ F_{s}(u), $ $ u\ge0 $  by means  of its Laplace transform
$\mathbb{F}_{s}(z):  $
\begin{align}                                            \label{opp5}
    \mathbb{F}_{s}(z)=
    \int_{0}^{\infty}e^{-uz}F_{s}(u)\,du =
    \frac{1}{z-c_{1}(s)}\,\frac{k_{1}(z)-s}{k_{2}(z)-s},
    \qquad \Re(z)> c_{2}(s).
\end{align}

\begin{theorem}  \label{topp1}

The Laplace transforms  of
$\underline\tau_{r}^{x}(b),$ $\overline\tau_{x}^{k}(b) $
and of the joint distribution of
$ \{\overline\tau_{x}^{k}(b),\overline T_{x}^{k}\} $  are such that for
$ s\ge0 $
\begin{align}                                     \label{opp6}
   \underline f_{r}^{x}(s)
&  =\frac{C_{1}^{b-x}(c_{2}(s),s)}{C_{1}^{b-r}(c_{2}(s),s)},
   \qquad r\le\min\{x,b\},
\end{align}
\begin{align}                                       \label{opp7}
   \overline f_{x}^{k}(s)
&   = 1+s\int_{0}^{b-x}R_{1}^{s}(u)\,du
    + s\int_{0}^{d_{2}}K_{b-x}^{s}(u)\,du-\notag\\
&   - \frac{K_{b-x}^{s}(d_{2})}{F_{s}(d_{2})}
    \left(\frac{s}{c_{1}(s)}+s\int_{0}^{d_{2}}F_{s}(u)\,du\right),
    \qquad k\ge\max\{x,b\},
\end{align}
\begin{align}                                     \label{opp8}
&   \overline f_{x}^{k}(z,s)
   = e^{zd_{2}}(k_{2}(z)-s)
    \left(\mathbb{K}_{b-x}^{s}(z)
    -\int_{0}^{d_{2}}e^{-uz}K_{b-x}^{s}(u)\,du\right)-\\
&   - \frac{K_{b-x}^{s}(d_{2})}{F_{s}(d_{2})}\,
     e^{zd_{2}}(k_{2}(z)-s)
    \left(\mathbb{F}_{s}(z)
    -\int_{0}^{d_{2}}e^{-uz}F_{s}(u)\,du\right),\quad k\ge\max\{x,b\}, \notag
\end{align}
where $ d_{2}=k-b,$  $ C_{i}^{x}(z,s)= e^{zx},$ $ x<0,$
$$
    C_{i}^{x}(z,s)=e^{zx}\left(1-(k_{i}(z)-s)
    \int_{0}^{x}e^{-uz}R_{i}^{s}(u)\,du\right),\quad x\ge 0.
$$
\end{theorem}

\begin{corollary}     \label{copp1}
Let   $ k_{1}(z)=k_{2}(z)=k(z). $
Then
\begin{align}                                           \label{opp9}
&   \underline f_{r}^{x}(s)
   =   e^{-(x-r)c(s)},\qquad r\le x,\notag\\
&   \overline f_{x}^{k}(s)
   =  1+s\int_{0}^{k-x}R^{s}(u)\,du
   - \frac{s}{c(s)}R^{s}(k-x),\qquad k\ge x,\\
&   \overline f_{x}^{k}(z,s)
   =  e^{(k-x)z}
   \left(1- (k(z)-s)\int_{0}^{k-x}e^{-uz}R^{s}(u)\,du\right)
   - R^{s}(k-x)\,\frac{k(z)-s}{z-c(s)},   \notag
\end{align}
where $ R^{s}(x),$ $x\ge0 $ is the resolvents of the process
$ \xi(t)=\xi_{i}(t); $ $ c(s)>0,$ $s>0$ is the unique root of the equation
$k(z)=s$ in the semi-plane $\Re(z)>0. $
\end{corollary}

\begin{corollary}     \label{copp2}

Assume  that the conditions  $ (A): $ $\bold E\xi_{i}(1)=0, $
$\sigma_{i}^{2}=\bold E\xi_{i}(1)^{2}<\infty$ are satisfied. Then
the following  limiting equalities  are valid:
\begin{align*}
&   \lim_{B\to\infty}
   \bold E\left[ e^{-s\overline\tau^{kB}_{xB}(bB)/B^{2}} \right]=
    \frac{\sigma_{1}e^{-(b-x)s_{1}}}
    {\sigma_{1}\ch((k-b)s_{2})+ \sigma_{2}\sh((k-b)s_{2})}, \qquad x\le b,\\
&   \lim_{B\to\infty}
   \bold E\left[ e^{-s\overline\tau^{kB}_{xB}(bB)/B^{2}} \right]=
    \frac{\sigma_{1}\ch((x-b)s_{2})+\sigma_{2}\sh((x-b)s_{2})}
    {\sigma_{1}\ch((k-b)s_{2})+\sigma_{2}\sh((k-b)s_{2})}, \qquad x\in[b,k];
\end{align*}
\begin{align*}
&   \lim_{B\to\infty}
   \bold E\left[ e^{-s\underline\tau^{xB}_{rB}(bB)/B^{2}} \right]=
    \frac{\sigma_{2}e^{-(x-b)s_{2}}}
    {\sigma_{1}\sh((b-r)s_{1})+\sigma_{2}\ch((b-r)s_{1})}, \qquad x\ge b,\\
&   \lim_{B\to\infty}
   \bold E\left[ e^{-s\underline\tau^{xB}_{rB}(bB)/B^{2}} \right]=
    \frac{\sigma_{1}\sh((b-x)s_{1})+\sigma_{2}\ch((b-x)s_{1})}
    {\sigma_{1}\sh((b-r)s_{1})+ \sigma_{2}\ch((b-r)s_{1})}, \qquad x\in[r,b],
\end{align*}
where $s_{i}=\sqrt{2s}/\sigma_{i},$ $ i=1,2, $  $k\ge\max\{x,b\},$
 $  r\le\min\{x,b\}. $
\end{corollary}

\begin{proof}
Let us  verify (\ref{opp6}).  Set  $ x=b. $
In view of  the definition  of the process $ \xi(x,t) $  (\ref{opp1}),
spatial homogeneity of  the processes $ \xi_{i}(t)  $ and  Markov property  of
$ \chi_{1,x}^{d_{1}} $ we can write the following equation:
 \begin{align}                                        \label{opp10}
   \underline f_{r}^{b}(s)=
   \underline{V}_{1,d_{1}}^{d_{1}}(s)+
   \int_{0}^{\infty}{V}_{1,d_{1}}^{d_{1}}(du,s)e^{-uc_{2}(s)}
   \underline f_{r}^{b}(s), \qquad d_{1}=b-r,
 \end{align}
where
$
  {V}_{1,x}^{d_{1}}(du,s)= \bold E\left[e^{-s\chi_{1,x}^{d_{1}}};
   T_{1}(x)\in du,\overline{A}_{1} \right],
$
$ x\in[0,d_{1}].$  It follows from  (\ref{opp2}),
(\ref{opp3}) that
\begin{align}                                      \label{opp11}
    \overline{V}_{1,x}^{d_{1}}(z,s)=
    C_{1}^{d_{1}-x}(z,s)-
    \frac{R_{1}^{s}(d_{1}-x)}{R_{1}^{s}(d_{1})}\,C_{1}^{d_{1}}(z,s).
\end{align}
Taking into account the latter equality and  (\ref{opp10}), we derive
$
  \underline f_{r}^{b}(s)= C_{1}^{d_{1}}(c_{2}(s),s)^{-1}.
$
Let  $ x>b. $ Then a
\begin{align*}
    \underline f_{r}^{x}(s)= e^{-(x-b)c_{2}(s)}\underline f_{r}^{b}(s)
    =\frac{e^{-(x-b)c_{2}(s)}}{C_{1}^{d_{1}}(c_{2}(s),s)}
    = \frac{C_{1}^{b-x}(c_{2}(s),s)}{C_{1}^{b-r}(c_{2}(s),s)}.
\end{align*}
If  $x\in[r,b],$  then   we have from
$ \underline f_{r}^{b}(s)=\underline f_{x}^{b}(s)\underline f_{r}^{x}(s) $
that
\begin{align*}
    \underline f_{r}^{x}(s)=
    \frac{\underline f_{r}^{b}(s)}{\underline f_{x}^{b}(s)}=
    \frac{C_{1}^{b-x}(c_{2}(s),s)}{C_{1}^{b-r}(c_{2}(s),s)},
    \qquad   x\in[r,b].
\end{align*}
Hence,  we  showed that (\ref{opp6}) is  valid for all  $ x\ge r. $

We now verify  (\ref{opp8}). Set  first $ x=b.$  Then taking into account the  defining formula  (\ref{opp1}) of  the process $ \xi(x,t), $  spatial homogeneity of the
processes $ \xi_{i}(t) $ and  Markov property of $ \tau_{1}^{+}(x), $
we can write
\begin{align*}
    \overline f_{b}^{k}(z,s)
&    =   e^{zd_{2}}\int_{d_{2}}^{\infty}m_{1}^{0}(du,s)e^{-uz}+
    \int_{0}^{d_{2}}m_{1}^{0}(du,s) \overline{V}_{2,u}^{d_{2}}(z,s)\notag \\
&    + \int_{0}^{d_{2}}m_{1}^{0}(du,s)
    \frac{R_{2}^{s}(d_{2}-u)}{R_{2}^{s}(d_{2})}\,\overline f_{b}^{k}(z,s),
    \qquad d_{2}=k-b,
\end{align*}
where
$
  m_{1}^{x}(du,s)=\bold E\left[e^{-s\tau_{1}^{+}(x)};T_{1}^{+}(x)\in du\right].
$
By means of this  equation we can determine the function
$ \overline f_{b}^{k}(z,s).$  Making  use of the  expression for the function  $ F_{s}(u), $
$ u\ge 0$ (\ref{opp5}), equalities  (\ref{opp2}), (\ref{opp3}),
after  performing some calculations,  we find
\begin{align}                                             \label{opp12}
    \overline f_{b}^{k}(z,s)
    = C_{2}^{d_{2}}(z,s)-\frac{R_{2}^{s}(d_{2})}{F_{s}(d_{2})}
     e^{zd_{2}}(k_{2}(z)-s)
    \left(\mathbb{F}_{s}(z)-\int_{0}^{d_{2}}e^{-uz}F_{s}(u)du\right).
\end{align}
Let $ x\in (b,k].$ Then  the function  $ \overline f_{x}^{k}(z,s) $
can be found from the following  equation:
\begin{align*}
    \overline f_{x}^{k}(z,s)
    = \overline{V}_{2,x-b}^{d_{2}}(z,s) +
    \frac{R_{2}^{s}(k-x)}{R_{2}^{s}(d_{2})}\,\overline f_{b}^{k}(z,s),
    \qquad  x\in (b,k].
\end{align*}
 In  view of (\ref{opp12})  we   derive
\begin{align}                                         \label{opp13}
    \overline f_{x}^{k}(z,s)
    = C_{2}^{k-x}(z,s)-\frac{R_{2}^{s}(k-x)}{F_{s}(d_{2})}\,
    \mathfrak{F}_{d_{2}}^{s}(z),\qquad  x\in [b,k],
\end{align}
where
$
    \mathfrak{F}_{d_{2}}^{s}(z)= e^{zd_{2}}(k_{2}(z)-s)
    \left(\mathbb{F}_{s}(z)-\int_{0}^{d_{2}}e^{-uz}F_{s}(u)du\right).
$
Let $ x<b. $  Then we can determine the function $ \overline f_{x}^{k}(z,s) $
from the following relation:
\begin{align*}
    \overline f_{x}^{k}(z,s)
    = e^{zd_{2}}\int_{d_{2}}^{\infty}m_{1}^{b-x}(du,s)e^{-uz}+
    \int_{0}^{d_{2}}m_{1}^{b-x}(du,s) \overline f_{u+b}^{k}(z,s).
\end{align*}
Employing (\ref{opp2}), (\ref{opp3}),the  definition of the function
$ K_{x}^{s}(u) $ (\ref{opp4}) and
the formula (\ref{opp13}), we obtain
\begin{align}                                         \label{opp14}
    \overline f_{x}^{k}(z,s)
    = \mathfrak{K}_{b-x}^{d_{2}}(z,s)
    -\frac{K_{b-x}^{s}(d_{2})}{F_{s}(d_{2})}\,
    \mathfrak{F}_{d_{2}}^{s}(z), \qquad x<b,
\end{align}
where
$$
    \mathfrak{K}_{b-x}^{d_{2}}(z,s)=
    e^{zd_{2}}(k_{2}(z)-s)\left(\mathbb{K}_{b-x}^{s}(z)-
    \int_{0}^{d_{2}}e^{-uz}K_{b-x}^{s}(u)\,du\right).
$$
Note that for $ x\in[b,k] $  it follows from the  definition of the function
$ K_{x}^{s}(u) $  (\ref{opp4}) that
\begin{align*}
    \mathfrak{K}_{b-x}^{d_{2}}(z,s)=C_{2}^{k-x}(z,s),
    \qquad K_{b-x}^{s}(d_{2}) = R_{2}^{s}(k-x).
\end{align*}
Hence, the  formula  (\ref{opp14}) is  valid for  all  $ x\le k.$
Since  $ \overline f_{x}^{k}(s)=\overline f_{x}^{k}(0,s),$  then
(\ref{opp7}) follows from  (\ref{opp8}) when $ z=0. $
We now verify   statements of   Corollary \ref{copp1}.
In case when $ k_{1}(z)=k_{2}(z)=k(z) $  we have
\begin{align*}
&   C_{i}^{x}(c_{2}(s),s)=e^{xc(s)},\quad \mathbb{F}_{s}(z)=(z-c(s))^{-1},
    \quad F_{s}(u)=e^{uc(s)}, \\
&   \mathbb{K}_{x}^{s}(z)=e^{xz}\int_{x}^{\infty}e^{-uz}R^{s}(u)\,du,
    \quad K_{x}^{s}(u)= R^{s}(x+u).
\end{align*}
These equalities and  (\ref{opp6})-(\ref{opp8})  imply the formulae
(\ref{opp9}). The limiting equalities \ref{copp2}  are derived in Section 4.
\end{proof}

\section { Exit from the interval by the process $ \xi(x,t).$}

For $B>0,$ $x,b\in[0,B],$   introduce the following random variable:
$$
  \chi_{x}(b)=\inf\{t:\,\xi(x,t)\notin[0,B]\},\qquad i=1,2
$$
i.e. the first exit time from the interval $[0,B]$ by the process  $\xi(x,t).$
Introduce  the events:
$\overline A=\{\xi(x,\chi_{x}(b))> B\} $  the process exits the interval
through the upper boundary;
$ \underline A=\{\xi(x,\chi_{x}(b))\le0\} $ the process exits the interval
through the lower boundary.
Denote  by
$
   T(x)=(\xi(x,\chi_{x}(b))-B)\bold I_{\overline A}
   +0\cdot \bold I_{\underline A}
$
the value of the overshoot at the instant of the first exit.
Define
\begin{align*}
   \underline{V}_{x}(s)=
   \bold E\left[e^{-s\chi_{x}(b)}; \underline A \right],\qquad
   \overline{V}_{x}(z,s)=
   \bold E\left[e^{-s\chi_{x}(b)-zT(x)};\overline A \right],
   \quad   \Re(z)\ge0.
\end{align*}

\begin{theorem}         \label{topp2}

The Laplace transforms of $ \chi_{x}(b), $ $x,b\in[0,B] $  and of the joint
distribution of $ \{\chi_{x}(b),T(x) \} $
are such that for $ s\ge0 $
\begin{align}                                             \label{opp15}
    \underline{V}_{x}(s)=
    \frac{K_{b-x}^{s}(B-b)}{K_{b}^{s}(B-b)},\qquad
    \overline{V}_{x}(s)=
    \mathfrak{K}_{b-x}^{B-b}(s)-
    \frac{K_{b-x}^{s}(B-b)}{K_{b}^{s}(B-b)}\:\mathfrak{K}_{b}^{B-b}(s),
\end{align}
where
$
   \overline{V}_{x}(s)= \bold E\left[e^{-s\chi_{x}(b)};\overline A \right],
$
\begin{align*}
    \mathfrak{K}_{x}^{u}(s)=
    1+s\int_{0}^{x}R_{1}^{s}(u)\,du+s\int_{0}^{u}K_{x}^{s}(v)\,dv,
    \quad x\in\mathbb{R},\;u\ge0;
\end{align*}
\begin{align}                                             \label{opp16}
   \overline{V}_{x}(z,s)=
   \mathfrak{K}_{b-x}^{B-b}(z,s)-
   \frac{K_{b-x}^{s}(B-b)}{K_{b}^{s}(B-b)}\:\mathfrak{K}_{b}^{B-b}(z,s),
\end{align}
and
\begin{align*}
    \mathfrak{K}_{x}^{u}(z,s)=
    e^{uz}\left(C_{1}^{x}(z,s)-
    (k_{2}(z)-s)\int_{0}^{u}e^{-vz}K_{x}^{s}(v)\,dv\right),
    \quad x\in\mathbb{R},\;u\ge0.
\end{align*}
\end{theorem}

\begin{corollary}     \label{copp3}

Assume that   $ k_{1}(z)=k_{2}(z)=k(z). $
Then  the following  equalities are  valid:
\begin{align}                             \label{opp17}
&    \underline{V}_{x}(s)=
    \frac{R^{s}(B-x)}{R^{s}(B)},\qquad
    \overline{V}_{x}(s)=
    C^{B-x}(s)-
    \frac{R^{s}(B-x)}{R^{s}(B)}\:C^{B}(s), \\
&    \overline{V}_{x}(z,s)=
    C^{B-x}(z,s)-
    \frac{R^{s}(B-x)}{R^{s}(B)}\:C^{B}(z,s),
\end{align}
where  $ C^{x}(s)=1+s\int_{0}^{x}R^{s}(u)\,du, $
$$
   C^{x}(z,s)= e^{xz}\left(1-(k(z)-s)\int_{0}^{x}e^{-uz}R^{s}(u)\,du\right),
$$
$ R^{s}(x), $  $x\ge0 $
are the resolvents  of the processes a  $ \xi(t)=\xi_{i}(t); $
$ c(s)>0,$ $s>0$  is the unique root of the equation  $ k(z)=s $ in the
semi-plane  $\Re(z)>0. $
\end{corollary}


\begin{corollary}         \label{copp4}

Assume, that $\bold E\xi_{i}(1)=0, $
$\sigma_{i}^{2}=\bold E\xi_{i}(1)^{2}<\infty,$ $ x,b\in(0,1).$
Then  the following  expansions  hold  as  $ B\to\infty $
\begin{align*}
&   \bold E\left[e^{-s\chi_{xB}(bB)/B^{2}};\,\underline A\right]\to
    \frac{\sigma_{1}\sh((b-x)s_{1})
    \ch(\overline{b} s_{2})+ \sigma_{2}\sh(\overline{b}s_{2})\ch((b-x)s_{1})}
    {\sigma_{1}\sh(bs_{1})\ch(\overline{b}s_{2})+
    \sigma_{2}\sh(\overline{b}s_{2})\ch(bs_{1})},\\
&   \bold E\left[e^{-s\chi_{xB}(bB)/B^{2}};\,\overline A\right]\to
    \frac{\sigma_{1}\sh(x s_{1})}
    {\sigma_{1}\sh(bs_{1})\ch(\overline{b}s_{2})+
    \sigma_{2}\sh(\overline{b}s_{2})\ch(bs_{1})},\quad x\in(0,b]  ;
\end{align*}
\begin{align*}
&   \bold E\left[e^{-s\chi_{xB}(bB)/B^{2}};\,\overline{A}\right]\to
    \frac{\sigma_{1}\sh(b s_{1})\ch((x-b)s_{2})+
    \sigma_{2}\sh((x-b)s_{2})\ch(b s_{1})}
    {\sigma_{1}\sh(bs_{1})\ch(\overline{b}s_{2})+
    \sigma_{2}\sh(\overline{b}s_{2})\ch(bs_{1})},\\
&   \bold E\left[e^{-s\chi_{xB}(bB)/B^{2}};\,\underline{A}\right]\to
    \frac{\sigma_{2}\sh (1-x) s_{2}}
    {\sigma_{1}\sh(bs_{1})\ch(\overline{b}s_{2})+
    \sigma_{2}\sh(\overline{b}s_{2})\ch(bs_{1})},\quad x\in[b,1),
\end{align*}
where $s_{i}=\sqrt{2s}/\sigma_{i},$ $ i=1,2, $ $\overline{b}=1-b. $
\end{corollary}

\begin{proof}
It is worth  noting that  the joint distribution of $\{\chi,T\} $
was found  in   \cite{2Ka2} for L\'{e}vy  processes  of general form.
To determine the Laplace transforms of this distribution, the  authors
used the one-boundary characteristics  $\{\tau^{x},T^{x}\}, $ $\{\tau_{x},T_{x}\}$
of the process. Following this approach, we derive the system of linear integral
equations with respect to the Laplace transforms $ \underline{V}_{x}(s),$
$
V_{x}(du,s)= \bold E\left[e^{-s\chi_{x}(b)};
T(x)\in du,\overline A \right]
$
\begin{align}                                                \label{opp19}
&   \bold E\left[ e^{-s\underline\tau_{0}^{x}(b)}\right]
    = \underline{V}_{x}(s)+
    \int_{0}^{\infty} V_{x}(du,s)
    \bold E\left[ e^{-s\underline\tau_{0}^{u+B}(b)}\right],
    \quad x,b\in[0,B],\notag\\
&   \bold E\left[ e^{-s\overline\tau_{x}^{B}(b)};
    \overline T_{x}^{B}\in du\right]= V_{x}(du,s)+
    \underline{V}_{x}(s)
    \bold E\left[ e^{-s\overline\tau_{0}^{B}(b)};
    \overline T_{0}^{B}\in du\right].
\end{align}
The first equation of this system means that the process $\xi(x,t) $
can reach the lower boundary $0$ either on the sample paths which do not cross
the upper boundary $ B, $ or on the sample paths which do cross the upper
boundary and then pass the lower boundary. The second equation is written
analogously. Observe, that the mathematical expectations which enter the
equations of the system are determined by (\ref{opp6})-(\ref{opp8}). Taking
into  account  the formulae (\ref{opp6}),
(\ref{opp19}), we derive
\begin{align}                                                \label{opp20}
&   \frac{C_{1}^{b-x}(c_{2}(s),s)}{C_{1}^{b}(c_{2}(s),s)}
    = \underline{V}_{x}(s)+
    \frac{e^{-c_{2}(s)(B-b)}}{C_{1}^{b}(c_{2}(s),s)}\,
    \overline{V}_{x}(c_{2}(s),s),\notag\\
&   \overline{f}_{x}^{B}(c_{2}(s),s)= \overline{V}_{x}(c_{2}(s),s)+
    \underline{V}_{x}(s)\overline{f}_{0}^{B}(c_{2}(s),s).
\end{align}
Formula  (\ref{opp8}) implies that
\begin{align*}
   \overline{f}_{x}^{B}(c_{2}(s),s)= C_{1}^{b-x}(c_{2}(s),s)e^{c_{2}(s)(B-b)}-
   \frac{K_{b-x}^{s}(B-b)}{F_{s}(B-b)}\,\tilde{F}(s),\quad x\in[0,B],
\end{align*}
where
$
   \tilde{F}(s)=(k_{1}(c_{2}(s)-s)(c_{2}(s)-c_{1}(s))^{-1}e^{c_{2}(s)(B-b)}.
$
Solving system (\ref{opp20}) with respect to two unknown functions
$ \underline{V}_{x}(s), $ $ \overline{V}_{x}(c_{2}(s),s), $ we find for all
$ x\in[0,B]$ that
\begin{align*}
&    \underline{V}_{x}(s)=
    \frac{K_{b-x}^{s}(d_{2})}{K_{b}^{s}(d_{2})},\\
&    \overline{V}_{x}(c_{2}(s),s)= e^{c_{2}(s)d_{2}}
    \left( C_{1}^{b-x}(c_{2}(s),s)-
   \frac{K_{b-x}^{s}(d_{2})}{K_{b}^{s}(d_{2})}\,
    C_{1}^{b}(c_{2}(s),s)\right),
\end{align*}
where  $ d_{2}=B-b. $
It follows from the second equation from the  system (\ref{opp19})
and from (\ref{opp8}) that
\begin{align*}
   \overline{V}_{x}(z,s) = \overline{f}_{x}^{B}(z,s)-
    \frac{K_{b-x}^{s}(d_{2})}{K_{b}^{s}(d_{2})}\:\overline{f}_{0}^{B}(z,s)
   =\mathfrak{K}_{b-x}^{d_{2}}(z,s)-\frac{K_{b-x}^{s}(d_{2})}{K_{b}^{s}(d_{2})}\:
    \mathfrak{K}_{b}^{d_{2}}(z,s),
\end{align*}
where
$$
    \mathfrak{K}_{x}^{d_{2}}(z,s)=e^{zd_{2}}\left( C_{1}^{x}(z,s)-
    (k_{2}(z)-s)\int_{0}^{d_{2}}e^{-uz}K_{x}^{s}(u)\,du\right),
    \quad x\in\mathbb{R}.
$$
The second equality (\ref{opp15}) can be derived from (\ref{opp16}) for $ z=0.$
If $ k_{1}(z)=k_{2}(z)=k(z), $ then
$$
   \mathfrak{K}_{x}^{d_{2}}(z,s)=C^{x+d_{2}}(z,s),\quad
   \mathfrak{K}_{x}^{d_{2}}(s)=1+s\int_{0}^{x+d_{2}}R^{s}(u)\,du,
   \quad x\in\mathbb{R}.
$$
The formulae (\ref{opp15}), (\ref{opp16}) of Theorem \ref{topp2} imply the
statements of Corollary  \ref{copp3}.
\end{proof}

\section{Asymptotic  behavior  }

In this section we assume that the following conditions are fulfilled $ (A):$
$\bold E\xi_{i}(1)=0, $ $\sigma_{i}^{2}=\bold E\xi_{i}(1)^{2}<\infty,$
$ i=1,2. $ It is  a well-known
fact (see for instance \cite{Bor3}, 
\cite{Kr18}, \cite{Sh}) that
\begin{align}                                               \label{opp21}
   \lim_{B\to\infty}\frac{1}{B}\,R_{i}^{s/B^{2}}(xB)=
   \frac{2}{\sigma_{i}\sqrt{2s}}\,\sh (x^{+} s_{i}),
   \quad \lim_{B\to\infty}Bc_{i}(s/B^{2})=s_{i},
\end{align}
where $ s_{i}=\sqrt{2s}/\sigma_{i},$ $ i=1,2, $  $ x^{+}=\max\{0,x\}. $
We now verify the limiting relations for the functions which appear in Theorems
\ref{topp1}, \ref{topp2}. Observe, that under the condition $ (A) $ the
following expansion is valid as $ B\to\infty, $ $ z>0 $
$$
    k_{i}(z/B)=\frac{1}{2}\sigma_{i}^{2}z^{2}/B^{2}+ o(B^{-2}).
$$
Then in view of the definition of the function
$ \mathbb{K}_{x}^{s}(z) $ (\ref{opp4}) we can write
\begin{align}                                              \label{opp22}
&    \tilde k_{x}^{s}(z)=
    \lim_{B\to\infty}\frac{1}{B^{2}}\,
    \mathbb{K}_{xB}^{s/B^{2}}(z/B)=
    \frac{e^{xz}}{\frac{1}{2}\,\sigma_{2}^{2}z^{2}-s},
    \quad x\le0,\notag  \\
&    \tilde k_{x}^{s}(z)=
    \frac{1}{\frac{1}{2}\,
    \sigma_{2}^{2}z^{2}-s}\left(\frac{z\sigma_{1}}{\sqrt{2s}}\,\sh(s_{1}x)
    +\ch(s_{1}x)  \right), \quad x\ge0.
\end{align}
For $ \Re(z)>\sqrt{2s}/\sigma_{2} $ the right-hand sides of these equalities
are the Laplace transforms:
$$
    \tilde k_{x}^{s}(z)=\int_{0}^{\infty}e^{-uz} k_{x}^{s}(u)\,du,
    \quad \Re(z)>\sqrt{2s}/\sigma_{2}.
$$
The formulae (\ref{opp22}) imply the following relation
\begin{align}                                               \label{opp23}
&    \lim_{B\to\infty}\frac{1}{B}\,
    K_{xB}^{s/B^{2}}(uB)= k_{x}^{s}(u)
    = \frac{1}{2\pi i}\int\limits_{\gamma-i\infty}^{\gamma+i\infty}
    e^{zu}\tilde k_{x}^{s}(z)\,dz=\notag\\
&   = \left\{
   \begin{array}{l}
    \frac{2}{\sigma_{2}\sqrt{2s}}\,\sh(x+u)^{+}s_{2}, \quad x\le0 , \\
    {} \\
    \frac{2}{\sigma_{2}^{2}\sqrt{2s}}
    \left(\sigma_{1}\sh(x s_{1})\ch (u s_{2})+
    \sigma_{2}\sh(u s_{2})\ch(x s_{1})\right), \quad x\ge0,
   \end{array}
   \right.
\end{align}
where $ \gamma> \sqrt{2s}/\sigma_{2}. $
Taking into account the latter equality, we can easily obtain the limiting
relations for the functions, which enter the statements
of Theorems \ref{topp1}, \ref{topp2}:
\begin{align}                                           \label{opp24}
   \lim_{B\to\infty}
    \mathfrak{K}_{xB}^{uB}(s/B^{2})
   = \left\{
   \begin{array}{l}
    \ch(x+u)^{+}s_{2}, \quad x\le0 , \\
    {} \\
    \frac{\sigma_{1}}{\sigma_{2}}\,\sh(xs_{1})\sh (u s_{2} )+
    \ch(u s_{2})\ch(xs_{1}), \quad x\ge0,
   \end{array}
   \right.
\end{align}
\begin{align}                                           \label{opp25}
   \lim_{B\to\infty}
    F_{s/B^{2}}(uB) =\frac{\sigma_{1}}{\sigma_{2}^{2}}\,
    (\sigma_{1}\ch(us_{2})+ \sigma_{2}\sh(us_{2}) ),\quad u\ge 0,
\end{align}
\begin{align}                                           \label{opp26}
   \lim_{B\to\infty}
    C_{1}^{xB}\left(c_{2}(s/B^{2}),s/B^{2}\right)
   = \left\{
   \begin{array}{l}
    e^{xs_{2}}, \quad x\le0 , \\
    {} \\
    \frac{\sigma_{1}}{\sigma_{2}}\,\sh(xs_{1})+ \ch(xs_{1}), \quad x\ge0.
   \end{array}
   \right.
\end{align}
We now verify the limiting equalities of Corollary  \ref{copp2}.
Let $k\ge\max\{x,b\},$  $x\le b. $  Then taking into  account  (\ref{opp7}),
(\ref{opp21}), (\ref{opp23}) and  (\ref{opp25}),  we have as  $B\to\infty $
\begin{align*}
&   \bold E\left[ e^{-s\overline\tau^{kB}_{xB}(bB)/B^{2}} \right]=
    \overline{f}_{xB}^{kB}(s/B^{2})\to
    \frac{\sigma_{1}}{\sigma_{2}}\,\sh(b-x)s_{1}\ch(d_{2}s_{2})+
    \ch(x-b)s_{1}\ch(d_{2}s_{2})\\
&   - \frac{\sigma_{1}\sh(d_{2}s_{2})+\sigma_{2}\ch(d_{2}s_{2})}
    {\sigma_{1}\ch(d_{2}s_{2})+\sigma_{2}\sh(d_{2}s_{2})}
    \left(\frac{\sigma_{1}}{\sigma_{2}}\,\sh((b-x)s_{1})\ch(d_{2}s_{2})+
    \ch((b-x)s_{1})\sh (d_{2}s_{2})\right)=\\
&   =\frac{\sigma_{1}e^{-(b-x)s_{1}}}
    {\sigma_{1}\ch((k-b)s_{2})+ \sigma_{2}\sh((k-b)s_{2})}, \qquad x\le b,
\end{align*}
where $ d_{2}=k-r. $  Similarly, we can derive the second formula of Corollary
\ref{copp2}:
\begin{align*}
&   \lim_{B\to\infty}
    \bold E\left[ e^{-s\overline\tau^{kB}_{xB}(bB)/B^{2}} \right]=
    \frac{\sigma_{1}\ch((x-b)s_{2})+\sigma_{2}\sh((x-b)s_{2})}
    {\sigma_{1}\ch((k-b)s_{2})+\sigma_{2}\sh((k-b)s_{2})}, \qquad x\in[b,k].
\end{align*}
Let $ r\le\min\{x,b\}.$  Then the following relation follows from (\ref{opp6})
and (\ref{opp26}) as  $B\to\infty $
\begin{align*}
   \bold E\left[ e^{-s\underline\tau^{xB}_{rB}(bB)/B^{2}} \right]&=
    \underline{f}_{rB}^{xB}(s/B^{2})=
    \frac{C_{1}^{(b-x)B}\left(c_{2}(s/B^{2}),s/B^{2}\right)}
    {C_{1}^{(b-r)B}\left(c_{2}(s/B^{2}),s/B^{2}\right)}\to \\
&   \to \frac{\sigma_{2}e^{(b-x)s_{2}}}
    {\sigma_{1}\sh((b-r)s_{1})+\sigma_{2}\ch((b-r)s_{1})}, \qquad x\ge b,
\end{align*}
\begin{align*}
   \bold E\left[ e^{-s\underline\tau^{xB}_{rB}(bB)/B^{2}} \right]
   \to \frac{\sigma_{1}\sh((b-x)s_{1})+\sigma_{2}\ch((b-x)s_{1})}
    {\sigma_{1}\sh((b-r)s_{1})+\sigma_{2}\ch((b-r)s_{1})}, \quad x\in[r,b].
\end{align*}
We now derive \ref{copp4}. The following relation follows from the first
formula of (\ref{opp15}) and from (\ref{opp23}) for $ x\in(0,b], $ as
$B\to\infty $
\begin{align*}
   \bold E & \left[e^{-s\chi_{xB}(bB)/B^{2}};\,\underline A\right]=
    \underline{V}_{xB}(s/B^{2})=
    \frac{K_{(b-x)B}^{s/B^{2}}(\overline{b}B)}
    {K_{bB}^{s/B^{2}}(\overline{b}B)}\to\\
&   \to \frac{\sigma_{1}\sh((b-x)s_{1})
    \ch(\overline{b}s_{2})+\sigma_{2}\sh(\overline{b}s_{2})\ch((b-x)s_{1})}
    {\sigma_{1}\sh(bs_{1})\ch(\overline{b}s_{2})+
    \sigma_{2}\sh(\overline{b}s_{2})\ch(bs_{1})},\quad x\in(0,b].
\end{align*}
where $ \overline{b}=1-b. $ Taking into account the second formula of
(\ref{opp15}) and (\ref{opp23}), (\ref{opp24}) we can write for $ x\in(0,b], $
as $ B\to\infty $
\begin{align*}
&   \bold E \left[e^{-s\chi_{xB}(bB)/B^{2}};\,\overline A\right]=
    \overline{V}_{xB}(s/B^{2})=
    \mathfrak{K}_{(b-x)B}^{\overline{b}B}(s/B^{2})-\\
&   -\frac{K_{(b-x)B}^{s/B^{2}}(\overline{b}B)}
    {K_{bB}^{s/B^{2}}(\overline{b}B)}\:
    \mathfrak{K}_{bB}^{\overline{b}B}(s/B^{2})\to
    \frac{\sigma_{1}\sh(x s_{1})}
    {\sigma_{1}\sh(bs_{1})\ch(\overline{b}s_{2})+
    \sigma_{2}\sh(\overline{b}s_{2})\ch(bs_{1})},\quad x\in(0,b].
\end{align*}
Analogously, the formulae of the corollary can be verified for $ x\in[b,1). $

\section {Appendix }
Let  $ \xi(t)\in\mathbb{R}, $ $ \xi(0)=0, $
$\bold E e^{-p\xi(t)}=e^{tk(p)},$ $\Re(p)=0$ be  a   general  L\'{e}vy process.  Denote  by
$$
    \xi_{t}^{+}=\sup_{u\le t}\xi(u),\qquad
    \xi_{t}^{-}=\inf_{u\le t}\xi(u)
$$
the running  supremum and infimum of the process.  For   $ x\ge0 $ define
$$
    \tau_{x}^{+}=\inf\{t>0:\xi(t)\ge x \},\qquad T_{x}^{+}=\xi(\tau_{x}^{+})-x
$$
the first crossing  time of a barrier  x and  the value  of the overshoot. Then  the following relation is valid
(\cite{PeRog}):
\begin{align}                                    \label{opp27}
   \int_{0}^{\infty}e^{-px}
   \bold E \left[e^{-s\tau_{x}^{+}-zT_{x}^{+}}\right]dx=
   \frac{1}{p-z}\left(1-\frac{\bold E e^{-p\xi_{\nu_s}^{+}}}
   {\bold E e^{-z\xi_{\nu_s}^{+}} }\right),\qquad \Re(p),\Re(z)\ge0,
\end{align}
where $\nu_{s}\sim\exp(s),$ $ s>0 $ is an exponential  random  variable  
independent  from the process $\xi(t).$
For  a   spectrally  positive   L\'{e}vy 
process with  Laplace exponent 
\begin{align}                                                 \label{opp28}
   k(z)=az+\frac{\sigma^{2} z^{2}}{2}+\int_{0}^{\infty}
   \left( e^{-zx}-1 +z \bold I_{\{0<x\le1\}}\right)\Pi(dx),
   \quad i=1,2
\end{align}
we have
$$
    \bold E e^{-z\xi_{\nu_s}^{+}}=\frac{s}{c(s)}\,\frac{p-c(s)}{k(p)-s},\quad
    \Re(p)\ge0,
$$
where $ c(s)>0,$ $ s>0 $ is the unique root of the equation  $k(z)-s=0,$
in the semi-plane $\Re(z)>0. $  It follows   from   the latter   relation and  
from (\ref {opp27})  that
\begin{align}                            \label{opp29}
   \int_{0}^{\infty}e^{-px}
   \bold E \left[e^{-s\tau_{x}^{+}-zT_{x}^{+}}\right]dx=
   \frac1{p-z}\left(1-\frac{p-c(s)}{k(p)-s}
   \;\frac{k(z)-s}{z-c(s)}\right).
\end{align}
Introduce the resolvents $ R^{s}(x),$ $x\ge0 $  \cite{SuSh} of the
spectrally one-sided  L\'{e}vy  process
$ \xi(t), $  $ t\ge0, $  by means of their Laplace transforms:
$$
    \int_{0}^{\infty}e^{-xz}R^{s}(x)\,dx = (k(z)-s)^{-1},
    \quad \Re(z)>c(s),\quad R^{s}(x)=0, \; x<0,
$$
Making use of the definition  of the  resolvent and inverting the Laplace 
transform with  respect  to $ p $ $ (\Re(p)>c(s))$  in both  sided of  
(\ref {opp29}),  we find
\begin{align*}
    \bold E \left[e^{-s\tau_{x}^{+}-zT_{x}^{+}}\right]=
    e^{xz}-R^{s}(x)\,\frac{k(z)-s}{z-c(s)}
    -(k(z)-s)e^{xz}\int_{0}^{x}e^{-uz}R^{s}(u)\,du,
\end{align*}
which is the  second  equality of  Lemma  1.

\end{document}